\documentclass{amsart}

\usepackage{a4wide}
\usepackage[numbers]{natbib}
\usepackage{mathtools}
\usepackage{enumitem}
\usepackage{tikz-cd}

\usepackage{pinlabel}
\usepackage{graphicx}

\newtheorem{theorem}{Theorem}[section]
\newtheorem*{theorem*}{Theorem}
\newtheorem{lemma}[theorem]{Lemma}
\newtheorem{prop}[theorem]{Proposition}

\newtheorem{cor}[theorem]{Corollary}

\theoremstyle{definition}
\newtheorem{definition}[theorem]{Definition}

\numberwithin{equation}{section}

\newcommand{\ho}[2]{H_{#1}(#2)}

\newcommand{\lbch}[4]{LCH^{\epsilon_{#1},\epsilon_{#2}}_{#3}(#4)}
\newcommand{\lbsch}[4]{LCH^{\epsilon^S_{#1},\epsilon^S_{#2}}_{#3}(#4)}
\newcommand{\lbcch}[4]{LCH_{\epsilon_{#1},\epsilon_{#2}}^{#3}(#4)}

\newcommand{\R}{\mathbb{R}}

\renewcommand{\L}{\Lambda}
\renewcommand{\epsilon}{\varepsilon}

\newcommand{\Z}{\mathbb{Z}}
\newcommand{\N}{\mathbb{N}}

\newcommand{\M}[3]{\mathcal{M}(#1;#2,\ldots,#3)}
\renewcommand{\Mc}[4]{\mathcal{M}_{#4}(#1;#2,\ldots,#3)}

\newcommand{\CE}[1]{\mathcal{A}(#1)}
\newcommand{\im}{\operatorname{im}}

\usepackage{fancyhdr}

\begin{document}
	\title{A note on geography of bilinearized Legendrian contact homology for disconnected Legendrian submanifolds}
	\author{Filip Strakoš}
	\address{Department of Mathematics, Geometry and Physics, Uppsala University,\newline Lägerhyddsvägen~1, 751 06 Uppsala, Sweden}
	\email{filip.strakos@math.uu.se}
	\urladdr{https://sites.google.com/view/filipstrakos/home}

	\begin{abstract}
		In this short note, we provide a criterion for DGA-homotopy of augmentations of Chekanov-Eliashberg algebra of disconnected Legendrian submanifolds. We apply the criterion to obtain the extension of geography results of Bourgeois and Galant concerning bilinearized Legendrian contact homology  to the case of disconnected Legendrian submanifolds.
	\end{abstract}
	\maketitle
	\section{Introduction}	
	
	For $M$ an $n$-dimensional smooth manifold we denote by $J^1(M)=T^*M\times \R$ its one-jet bundle. We endow it with a canonical contact structure given by the kernel of the co-oriented one-form $dz-\eta$, where $\eta$ is the Liouville one form on $T^*M$. To study the Legendrian isotopy classes of such Legendrians we can define the Chekanov-Eliashberg algebra $(\mathcal{A}(\Lambda),\partial)$ for generic closed Legendrian submanifold $\L$ of $J^1(M)$ (see \cite{ChekanovLCH}, \cite{EEScontacthomology}). We will consider Legendrian submanifolds whose Maslov class vanishes to obtain $\Z$-graded differential graded algebra. Where differential $\partial$ counts rigid pseudoholomorphic disks in the symplectization of $J^1(M)$.
	
	The homology of $\mathcal{A}(\Lambda)$ is hard to work with and so we (bi)linearize the differential using augmentations $\epsilon:(\mathcal{A},\partial)\to (Z_2,0)$ (see \cite{ChekanovLCH},\cite{EEScontacthomology} for linearization and \cite{bil} for bilinearization). We denote by $LCH^\epsilon(\Lambda)$ and $LCH^{\epsilon_1,\epsilon_2}(\Lambda)$ the linearized and bilinearized Legendrian contact homology respectively.
	
	In this context, the question of the DGA-homotopy of augmentations of the Chekanov-Eliashberg algebra naturally appears. In \cite{bil} Bourgeois and Chantraine proved that the cardinality of the set $$\mathcal{E}(\Lambda)=\{[\epsilon]_\sim \, \,|\,\epsilon:(\CE{\Lambda},\partial)\to (\Z_2,0) \text{ is an augmentation of } \CE{\Lambda}\}$$ of DGA-homotopy classes of augmentations is a Legendrian isotopy invariant. It is not a simple task to decide whether two given augmentations $\epsilon_1$ and $\epsilon_2$ belong to the same the DGA-homotopy class. Therefore, it is surprising that there exist criteria that can be used to distinguish those classes.
	
	The first hints of the existence of such a criterion can be tracked to work on duality long exact sequence of Ekholm, Etnyre and Sabloff (see \cite{dualityseq}). Combining those with results concerning bilinearized Legendrian contact homology of Bourgeois and Chantraine (see \cite{bil}) leads to a necessary condition that reads:
	\begin{equation*}\label{odkaz na jedn}
		\text{If } \epsilon_1\text{ and } \epsilon_2 \text{ are DGA-homotopic, then } \tau_{0} \text{ vanishes.}
	\end{equation*}
	Here $\tau_0:LCH^{\epsilon_1,\epsilon_2}_0(\Lambda)\to H_0(\L)$ is the map from the duality long exact sequence relating bilinearized Legendrian contact homology and the Morse homology of a connected Legendrian submanifold. This condition was found to be also sufficient by Bourgeois and Galant in \cite{geog} for connected Legendrian submanifolds. 
	
	For disconnected Legendrian submanifolds this condition fails. One can easily find a Legendrian link with non-vanishing $\tau_0$ arising from the duality sequence for two DGA-homotopic augmentations. This happens because the condition does not pass to the chain level anymore.
	
	However, in the connected case, one can show using the duality that $\tau_{0}$ vanishes if and only if $\tau_n:LCH^{\epsilon_1,\epsilon_2}_n(\Lambda)\to H_n(\L)$ is surjective, in particular, if $\tau_n$ hits the fundamental class of the Legendrian since it is connected. Therefore, we restate and prove the condition so that it holds even in the disconnected case:
	
	\begin{theorem}\label{main thm}
		Let $M$ be an $n$-dimensional smooth manifold and $\Lambda$ for be a closed Legendrian submanifold of $(J^1(M),dz-\eta)$ with vanishing Maslov class. Denote by $[\Lambda]$ its fundamental class. Let  $\epsilon_1,\epsilon_2$ be two augmentations of the Chekanov-Eliashberg algebra $\CE{\Lambda}$ over $\Z_2$. Then the following holds:
		\begin{equation}
			\epsilon_1 \text{ and } \epsilon_2 \text{ are DGA homotopic } \,\Leftrightarrow [\Lambda] \text{ is an  element of the image of } \tau_{n} .
		\end{equation}
	\end{theorem}
	
	There is one possible interpretation of our result in the context of exact Lagrangian fillings of Legendrian submanifolds due to Ekholm, Honda, Kálmán and Karlsson
	\begin{theorem*} [\cite{EHKmorphism},  \cite{Karlssonorientation}]
		An exact Lagrangian filling $L$ of a closed Legendrian submanifold $\Lambda$ induces an augmentation $\epsilon_{L}:(\mathcal{A}(\Lambda),\partial)\to (\Z_2,0)$. If $L_1$ and $L_2$ are two exact Lagrangian fillings of $\L$ that are isotopic through exact Lagrangian fillings, then $\epsilon_{L_1}\sim \epsilon_{L_2}$.
	\end{theorem*}
	Therefore, we immediately obtain the following necessary condition.
	\begin{cor}
		If $L_1$ and $L_2$ are two exact Lagrangian fillings of $\L$ that are isotopic through exact Lagrangian fillings, then $[\Lambda]$ is an element of the image of $\tau_{-,n}$.
	\end{cor}
	
	Nevetheless, the main application of our result will lie in the geography of bilinearized Legendrian contact homology for disconnected Legendrian submanifolds. In other words, we are asking about what polynomials can be attained as Poincaré polynomials of bilinearized Legendrian contact homology, that is $$ P_{\L,\epsilon_1,\epsilon_2}(t)=\sum_{k\in \Z} \dim_{\Z_2} LCH^{\epsilon_1,\epsilon_2}_k(\L)\,\, t^k, $$ and vice versa, which Legendrian submanifolds realize a particular admissible polynomial. This question was fully answered for linearized Legendrian contact homology by Bourgeois, Sabloff and Traynor in \cite{gfgeography} as it was observed in \cite{bil}. And because of the results of Bourgeois and Chantraine \cite{bil} it is enough to describe the geography when $\epsilon_1\not\sim \epsilon_2$. That was done in \cite{bil} by Bourgeois and Galant for the connected case. In this note, we extend the results to the disconnected case. More specifically, we define a version of bLCH-admissible polynomials for disconnected Legendrian submanifolds (links in dimension three) called lbLCH-admissible polynomials (see Definition \ref{def: admissible polyn}) and we prove:
	
	\begin{theorem}\label{geography1}
		Let $\L$ be a Legendrian submanifold of $J^1(M)$, $\dim M = n $, with vanishing Maslov class, that consists of $r$ $n$-dimensional components for any natural number $r$, and $\epsilon_1$, $\epsilon_2$ be two DGA-non-homotopic augmentations. Then $P_{\L,\epsilon_1,\epsilon_2}$ is a lbLCH-admissible polynomial.
	\end{theorem}
	
	and to complete the geography also the other direction:
	
	\begin{theorem}\label{geography2}
		Let $M$ be a smooth manifold of dimension $n$. If $P\in \N_0[t,t^{-1}]$ is any lbLCH-admissible polynomial, then there is $\L$ a Legendrian submanifold of $J^1(M)$ whose connected components are connected Legendrian submanifolds and two DGA-non-homotopic augmentations $\epsilon_1,\epsilon_2$ of its Chekanov-Eliashberg algebra so that $P_{\L,\epsilon_1,\epsilon_2}=P$. 
	\end{theorem}
	\subsection*{Acknowledgements} The results found in this short note were produced under the supervision of my Erasmus+ Intership supervisor Frédéric Bourgeois and my Master Thesis supervisor Roman Golovko, both of whom I am greatly indebted to.  Moreover, I would like to express
	gratitude to the Institut de Mathématique d'Orsay for the financial support and the Team of Topology and
	Dynamics for their hospitality and welcoming environment.
	\newpage
	\section{Background}
	\subsection{(Bi)linearization}  
	
	\begin{definition}
		Let $\epsilon_1,\epsilon_2$ be augmentations of $\mathcal{A}(\Lambda)$. A linear map $K:\mathcal{A}(\Lambda)\to \Z_2$ satisfying $K(ab)=\epsilon_1(a) K(b)+K(a)\epsilon_2(b)$ for all $a,b\in \mathcal{A}(\Lambda)$ is called a $(\epsilon_1,\epsilon_2)$-antiderivation.
		If it exists and $\epsilon_1-\epsilon_2=K\circ \partial$, then the augmentations are said to be DGA-homotopic, notation $\epsilon_1\sim\epsilon_2$. 
	\end{definition}

	Note that the relation $\sim$ above is an equivalence and that each augmentation is uniquely determined by its values on Reeb chords of $\Lambda$ and so for compact Legendrian $\Lambda$ we have a finite set$$\mathcal{E}(\Lambda)=\{[\epsilon]_\sim \, \,|\,\epsilon:(\CE{\Lambda},\partial)\to (\Z_2,0) \text{ is an augmentation of } \CE{\Lambda}\}.$$
	
	\begin{theorem}[Theorem 1.3. in \cite{bil}]\label{thm: invariance of set of augmentations}
		Let $\{\Lambda_t\}_{t\in [0,1]}$ is a Legendrian isotopy, then we have a bijection of $\mathcal{E}(\Lambda_0)$ and $\mathcal{E}(\Lambda_1)$.
	\end{theorem}

	\begin{theorem}[Theorem 1.4, \cite{bil}]
		Let $\Lambda$ be a compact generic Legendrian submanifold of $J^1(M)$ with vanishing Maslov number, then if $\epsilon_1,\epsilon_2$ are two DGA-homotopic augmentations of $\CE{\Lambda}$, then $LCH^{\epsilon_1}(\Lambda)\cong LCH^{\epsilon_2}(\Lambda)$.
	\end{theorem}
	
	Therefore, the cardinality of the set $\mathcal{E}(\Lambda)$ is a Legendrian isotopy invariant. However, for $\{\Lambda_t\}_{t\in [0,1]}$ a Legendrian isotopy and $f:\mathcal{E}(\Lambda_0)\to \mathcal{E}(\Lambda_1)$ be the bijection from Theorem \ref{thm: invariance of set of augmentations}, then it is true that $LCH^{\epsilon}(\Lambda_0)\simeq LCH^{f(\epsilon)}(\Lambda_1)$.
	
	\begin{theorem}[Theorem 1.2, \cite{bil}]\label{thm: set of linearized lch}
		Let $\Lambda$ be a compact generic Legendrian submanifold of $J^1(M)$. Consider the set
		\begin{equation*}
			\mathcal{H}(\Lambda)=\bigcup_{[\epsilon]_\sim \in \mathcal{E}(\Lambda)}\{ LCH^\epsilon(\Lambda)\}.
		\end{equation*}
		Let $\{\Lambda_t\}_{t\in [0,1]}$ is a Legendrian isotopy, then the sets $\mathcal{H}(\Lambda_0)$ and $\mathcal{H}(\Lambda_1)$ coincide.
	\end{theorem}
	
	Analogously to linearized Legendrian contact homology we have the following.
	\begin{theorem}[Theorem 1.2, \cite{bil}]
		Let $\Lambda$ be a compact generic Legendrian submanifold of $J^1(M)$ with vanishing Maslov number, then if $\epsilon_1,\epsilon_2,\epsilon$ are augmentations of $\CE{\Lambda}$, where $\epsilon_1\sim \epsilon_2$, then $LCH^{\epsilon_1,\epsilon}(\Lambda)\cong LCH^{\epsilon_2,\epsilon}(\Lambda)$ and $LCH^{\epsilon,\epsilon_1}(\Lambda)\cong LCH^{\epsilon,\epsilon_2}(\Lambda)$.
	\end{theorem}
	In particular,
	\begin{cor}\label{cor: dga homot of augm and blch}
		Let $\Lambda$ be a compact generic Legendrian submanifold of $J^1(M)$ with vanishing Maslov number, then if $\epsilon_1,\epsilon_2$ are augmentations of $\CE{\Lambda}$, where $\epsilon_1\sim \epsilon_2$, then $LCH^{\epsilon_1,\epsilon_2}(\Lambda)\cong LCH^{\epsilon_2,\epsilon_2}(\Lambda)=LCH^{\epsilon_{2}}(\Lambda)$.
	\end{cor}
	\begin{theorem}[Theorem 1.2, \cite{bil}]
		Let $\Lambda$ be a compact generic Legendrian submanifold of $J^1(M)$. Consider the set
		\begin{equation*}
			\mathcal{H}^b(\Lambda)=\bigcup_{([\epsilon_1]_\sim,[\epsilon_2]_\sim) \in \mathcal{E}(\Lambda)\times\mathcal{E}(\Lambda)}\{ LCH^{\epsilon_1,\epsilon_2}(\Lambda)\}.
		\end{equation*}
		Let $\{\Lambda_t\}_{t\in [0,1]}$ is a Legendrian isotopy, then the sets $\mathcal{H}^b(\Lambda_0)$ and $\mathcal{H}^b(\Lambda_1)$ coincide.
	\end{theorem}
	
	In view of Theorem \ref{thm: set of linearized lch} and Corollary \ref{cor: dga homot of augm and blch}, we see that $\mathcal{H}(\Lambda)\subset\mathcal{H}^b(\Lambda)$. And so bilinearized Legendrian contact homology is a stronger invariant of Legendrian isotopy which encodes the non-commutativity of the Chekanov-Eliashberg algebra that is lost in the process of linearization.
	
	In general, it is not an easy task to determine the DGA-homotopy class of a augmentation of $\CE{\Lambda}$ computationally and thus determining the cardinality of $\mathcal{E}(\Lambda)$ is an interesting problem. In the standard $3$-dimensional space the cardinality of $\mathcal{E}(\Lambda)$ was studied by Ng, Rutherford, Shende, Sivek in \cite{Ng_2017}.
	
	It is quite surprising that there is a criterion for DGA-homotopy of augmentations in any dimension. This first appeared in work of Bourgeois and Galant (see \cite{geog}). 
	
	\begin{theorem}[Proposition 3.3. in \cite{geog}]\label{thm: criterBG}
		Let $\Lambda$ be a connected compact generic Legendrian submanifold of $J^1(M)$ has dimension with vanishing Maslov number, where $M$ has dimension $n$. Then if $\epsilon_1,\epsilon_2$ are augmentations of $\CE{\Lambda}$ it holds that \begin{equation*}
			\dim_{\Z_2} LCH^{\epsilon_2,\epsilon_1}_n(\Lambda)-\dim_{\Z_2} LCH^{\epsilon_1,\epsilon_2}_{-1}(\Lambda)=\begin{cases}
				0, & \, \epsilon_1\not\sim \epsilon_2,\\
				1, & \, \epsilon_1  \sim \epsilon_2.
			\end{cases}
		\end{equation*}
		Equivalently,
		\begin{equation*}
			\epsilon_1\sim\epsilon_2 \iff \tau_0=0.
		\end{equation*}
	\end{theorem}
	
	This criterion is an important step in proving so-called geography result for bilinearized Legendrian contact homology.

	\subsection{Duality long exact sequence}\label{sec:dualityseq}
	
	We will denote by $\Lambda$ a disconnected Legendrian submanifold, that is $\Lambda=\coprod_{j=1}^r\Lambda_j$ for some $r\in \mathbb{N}$, where $\L_j$ are connected components. Moreover, we suppose that $\Lambda$ has vanishing Maslov number.
	
	Theorem \ref{thm: duality long exact sequence} below directly follows from Theorem 1.1 in \cite{dualityseq} that was originally proven for linearized Legendrian contact homology by Ekholm, Etnyre, Sabloff. Nevertheless, both the statement and the proof translates to the bilinearized setting as was observed in \cite{bil} .
	
	\begin{theorem}[Theorem 1.1 in \cite{dualityseq}]\label{thm: duality long exact sequence}
		Let $M$ be a smooth manifold of dimension $n$ and  $\Lambda\subset J^1(M)$ be a closed chord generic Legendrian submanifold such that we can completely displace the Lagrangian projection $\Pi_{T^*M}(\Lambda)$ from itself using a Hamiltonian isotopy, that is, if $\phi_t$ is a Hamiltonian isotopy of $T^*M$ for $t\in [0,1]$, then $\phi_1(\Pi_{T^*M}(\Lambda))\cap \Pi_{T^*M}(\Lambda)=\emptyset$. Moreover, we assume that we have $\epsilon_1,\epsilon_2$ two augmentations of $\CE{\Lambda}$ over $\Z_2$, then we obtain a long exact sequence
		\begin{equation}\label{eq:dualsequnce}
			\dots \to \lbcch{2}{1}{n-k-1}{\Lambda}\to \lbch{1}{2}{k}{\Lambda}\xrightarrow{\tau_k}\ho{k}{\Lambda}\xrightarrow{\sigma_{n-k}}\lbcch{2}{1}{n-k}{\Lambda}\to \cdots,
		\end{equation}  
		where $\ho{\bullet}{\Lambda}$ is the Morse (or equivalently singular) homology of $\Lambda$ with coefficients in $\Z_2$.
		This long exact sequence is to be called the duality long exact sequence of $\Lambda$. 
	\end{theorem}
	
	Since the detailed description of this sequence is beyond the scope of this paper, we present rather informal description below. We refer the reader to \cite{dualityseq} for details.
	
	The map $\tau_k$ above counts so called generalized lifted disks $(u,\gamma)$, where $u$ is a punctured pseudo-holomorphic as in Section 2.2.3 of \cite{dualityseq} and $\gamma$ is a negative gradient flow-line of some fixed perturbing Morse function $f:L\to \R$ which connects a critical point $q$ of $f$ with index $k$ to some generic point $p$ on the boundary of $u$ and the orientation of the flow is oriented towards $c$ (that is $f(p)>f(q)$).  
	
	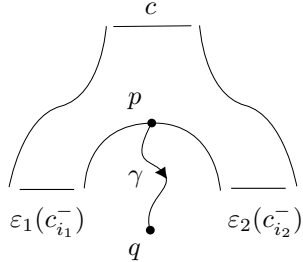
\begin{figure}[!h]
		\centering
		\setlength{\unitlength}{0.1\textwidth}
		\begin{picture}(10,2.5)
			\put(3.5,0){\includegraphics[scale=1]{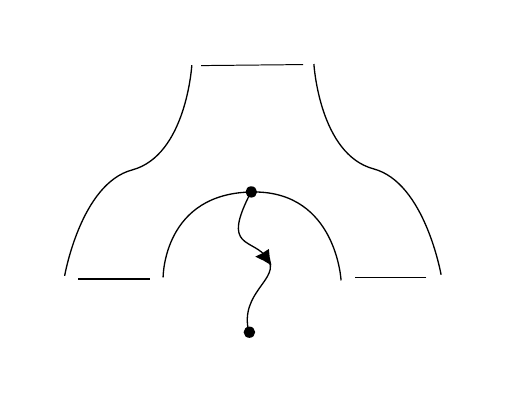}}
			\put(3.95,0.5){$\epsilon_1(c^-_{i_1})$}
			\put(5.89,0.5){$\epsilon_2(c^-_{i_2})$}
			\put(5.15,2.4){$c$}
			\put(5.00,1.6){$p$}
			\put(5.00,0.25){$q$}
			\put(5.00,0.9){$\gamma$}
		\end{picture}
		\caption{Example of a lifted generalized disk contributing to $\tau_{k}$.}\vspace{-3.0 pt}
		\label{lifted generalized disk}
	\end{figure}
	
	Denote the moduli space of suitable rigid  $(u,\gamma)$ where the boundary point is in between negative punctures $c_{i_{l-1}}$ and $c_{i_{l}}$ as $\mathcal{M}_{c^+,\textbf{c},q}$. Then we have
	\begin{equation*}
		\tau_k(c)=\sum_{\dim\mathcal{M}(c;\textbf{c}^-)/\R=\operatorname{Index}_f(q)-1} \#_2\mathcal{M}_{c,\textbf{c},q}\sum_{j=1}^{k}\epsilon_1(c_{i_1})\ldots\epsilon_1(c_{i_{l-1}})\, q \,\epsilon_2(c_{i_l})\ldots\epsilon_2(c_{i_k}).
	\end{equation*}
	
	On the other hand, the map $\sigma_{k}$ counts lifted generalized disks as above, however, the flow-line heads in the opposite direction (that is $f(p)<f(q)$).
	
	And so
	\begin{equation*}
		\sigma_{k}(q)=\hspace{-4.0 pt}\sum_{\dim\mathcal{M}(c;\textbf{c}^-)/\R=\operatorname{Index}_f(q)-1}\hspace{-4.0 pt} \#_2\mathcal{M}_{c,\textbf{c},q}\sum_{j=1}^{k}\epsilon_1(c_{i_1})\ldots\epsilon_1(c_{i_{l-1}})\, c \,\epsilon_2(c_{i_l})\ldots\epsilon_2(c_{i_k}).
	\end{equation*}
	
	The main difference between the linearized and bilinearized case is that the one has to pay attention to the ordering of augmentations in the duality formula from which the name of the sequence stems. More precisely, consider $\epsilon_1, \epsilon_2$ two augmentations of $\CE{\Lambda}$ of some $\Lambda=\coprod_{j=1}^r\Lambda_j$ as above. We have two duality sequences: fist for the ordering $(\epsilon_1,\epsilon_2)$, that we call positive:
	\begin{equation}
		\cdots\to \lbcch{2}{1}{n-k-1}{\Lambda}\to \lbch{1}{2}{k}{\Lambda}\xrightarrow{\tau_{+,k}}\ho{k}{\Lambda}\xrightarrow{\sigma_{+,n-k}}\lbcch{2}{1}{n-k}{\Lambda}\to \cdots
	\end{equation}
	second for the ordering $(\epsilon_2,\epsilon_1)$, that we call negative:
	\begin{equation}
		\cdots \to\lbcch{1}{2}{n-k-1}{\Lambda}\to \lbch{2}{1}{k}{\Lambda}\xrightarrow{\tau_{-,k}}\ho{k}{\Lambda}\xrightarrow{\sigma_{-,n-k}}\lbcch{1}{2}{n-k}{\Lambda}\to \cdots
	\end{equation}
	
	Recall that since all components $\Lambda_j$ for $j=1,\ldots,r$ of our Legendrian $\Lambda$ are closed we have an intersection pairing $\bullet:\ho{k}{\L_{j}}\otimes\ho{n-k}{\L_j}\to\Z_2$ for each $k\in \Z$. Now define the intersection pairing on $\Lambda$ for $k\in \Z$ and for $c=(c_1,\ldots,c_r)\in\ho{k}{\Lambda}=\ho{k}{\L_{1}}\oplus\ldots\oplus\ho{k}{\L_r}$ and $d=(d_1,\ldots,d_r)\in \ho{n-k}{\Lambda}=\ho{n-k}{\L_{1}}\oplus\ldots\oplus\ho{n-k}{\L_r}$ to be \begin{equation}
		c \bullet d =\sum_{j=1}^{r} c_j\bullet d_j.
	\end{equation}
	For precise definition of the pairing $\bullet$ on the components of the disconnected Legendrian submanifold $\Lambda$ see Section 3.3.3. of \cite{dualityseq}.
	
	Define a pairing $\langle\cdot,\cdot\rangle:C(\mathcal{R}(\Lambda))^*\otimes_{\Z_2}C(\mathcal{R}(\Lambda))\to \Z_2$ on generators $\{c_i: i\in \{1,\ldots, \#\mathcal{R}(\Lambda)\}\}$ as follows
	\begin{equation*}
		\langle c^*_{i},c_j\rangle=\begin{cases}
			1, & \text{ if } i=j,\\
			0, & \text{ otherwise.}
		\end{cases}
	\end{equation*}
	
	Note that this precisely corresponds to evaluation $\langle c^*_{i},c_j\rangle=c^*_j(c_i)$. In particular, if $|c^*_i|\neq|c_j|$, then $\langle c^*_i,c_j\rangle=0$.
	
	

	Consider $\partial^{\epsilon_1,\epsilon_2}$ the differential on $C(\mathcal{R}(\Lambda))$ and $\mu_{\epsilon_2,\epsilon_1}^1$ on $C(\mathcal{R}(\Lambda))^*$ the dual differential to $\partial^{\epsilon_1,\epsilon_2}$, that is $$\langle \mu_{\epsilon_2,\epsilon_1}^1(a),b\rangle=\langle a,\partial(b)\rangle,$$ for $a\in C(\mathcal{R}(\Lambda))$  and $b\in \mu_{\epsilon_2,\epsilon_1}^1$. Moreover, consider the tensor product $W=C(\mathcal{R}(\Lambda))^*\otimes_{\Z_2}C(\mathcal{R}(\Lambda))$ endowed with the standard differential $$d_W(c^*\otimes d)=\mu_{\epsilon_2,\epsilon_1}^1(c^*)\otimes d+c^*\otimes \partial^{\epsilon_2,\epsilon_1}(d)$$ for $c^*\in C(\mathcal{R}(\Lambda))^*$ and $d\in C(\mathcal{R}(\Lambda))$. Using the generalized Künneth formula we obtain that
	\begin{equation*}
		H_\bullet(W,d_W)\cong \bigoplus_{k+l=\bullet} H_k(C(\mathcal{R}(\Lambda))^*,\mu_{\epsilon_2,\epsilon_1}^1)\otimes_{\Z_2} H_l(C(\mathcal{R}(\Lambda)),\partial^{\epsilon_2,\epsilon_1}), 
	\end{equation*}
	where  torsion does not occur for we work over the field $\Z_2$.
	
	Define $F:W\to \Z_2$ by $F(a)=\sum_{i}\langle c_i,d_i\rangle$ for $a=\sum_i c_i\otimes_{\Z_2} d_i \in W$, then $F\circ d_W=0$ since the differentials are dual to each other. Therefore $F$ descends to homology of $(W,d_W)$ and so does the pairing.
	
	We will denote the pairing that we have just constructed as $$\langle\cdot,\cdot\rangle_+:LCH_{\epsilon_{2},\epsilon_1}(\Lambda)\otimes_{\Z_2} LCH^{\epsilon_{1},\epsilon_{2}}(\Lambda)\to \Z_2,$$ and for the opposite order of augmentations we will write
	$$\langle\cdot,\cdot\rangle_-:LCH_{\epsilon_{1},\epsilon_2}(\Lambda)\otimes_{\Z_2} LCH^{\epsilon_{2},\epsilon_{1}}(\Lambda)\to \Z_2.$$
	
	\begin{prop}\label{prop: almost non-degeneracy}
		Then for any non-zero class $[a]\in LCH_{\epsilon_{1},\epsilon_2}^{k}(\Lambda)$, there is a Reeb chord $c\in \mathcal{R}(\Lambda)$ so that $\langle [a],[c]\rangle_-\neq 0$, where the pairing
		$$\langle \cdot, \cdot\rangle_-:LCH^k_{\epsilon_1,\epsilon_2}(\Lambda)\otimes_{\Z_2}LCH_k^{\epsilon_2,\epsilon_1}(\Lambda)\to \Z_2$$ is as above. 
	\end{prop}
	\begin{proof}
		Let $a=\sum_{i=1}^m c^*_i+db$ be a representative of the class $[a]$, where $c_i\in \mathcal{R}(\Lambda)$ are all distinct, non-exact, and $|c^*_i|=k$, and moreover, $b\in C_{k+1}^*(\Lambda)$. Now because the class $[a]$ is non-zero then $m>0$. If we had that
		$\langle [a],[c]\rangle_-=0$ for all $c\in \mathcal{R}(\Lambda)$ of grading $k$, then \begin{align*}
			0&=\langle [a], [c] \rangle_-=\langle a, c \rangle=\langle \sum_{i=1}^m c^*_i+db, c\rangle=\langle \sum_{i=1}^m c^*_i , c \rangle=\sum_{i=1}^{m}c^*_i(c),
		\end{align*}
		therefore, $m=0$ which is a contradiction.
	\end{proof}
	
	We have an analogue of Proposition 3.9  in \cite{dualityseq}.

	\begin{prop}\label{signs switching}
		The pairs of maps $\tau_{+,k}$ and $\sigma_{-,k}$, and $\tau_{-,k}$ and $\sigma_{+,k}$ are adjoint in the following sense: 
		
		Let us have $c\in \ho{k}{\Lambda}$ and a chord $q$ of grading $n-k$
		\begin{align*}
			\langle\sigma_{-,n-k}(c),[q]\rangle_-&= c \bullet \tau_{+,n-k}([q]),\\
			\langle\sigma_{+,n-k}(c),[q]\rangle_+&= c \bullet \tau_{-,n-k}([q]),
		\end{align*}
		
		where $\bullet: \ho{k}{\Lambda}\otimes\ho{n-k}{L}\to \Z_2$ is the intersection pairing, and pairings
		\begin{align*}
			\langle\cdot,\cdot \rangle_-&:\lbcch{2}{1}{n-k}{\Lambda}\otimes_{\Z_2} \lbch{2}{1}{n-k}{\Lambda}\to \Z_2,\\
			\langle \cdot,\cdot\rangle_+&:\lbcch{1}{2}{n-k}{\Lambda}\otimes_{\Z_2}\lbch{1}{2}{n-k}{\Lambda}\to \Z_2
		\end{align*} are as above.

	\end{prop}
	\begin{proof}
		Let us prove the first equation the other case is analogous.
		
		The right handed side counts holomorphic disks with $q$ mixed positive puncture and $c$ negative Morse puncture and with possibly other augmented negative punctures between $q$ and $c$ with $\epsilon_1$ and between $c$ and $q$ with $\epsilon_2$. Now the bijective correspondence from Theorem 3.6 form \cite{dualityseq} implies that this disks corresponds to the lifted generalized disk with $\gamma$ a negative gradient flow line of the perturbing function $f$ ending at $c$ and connecting it to the boundary of the disk.

		To pass to the right side that is from homology to cohomology we change the sign of the perturbing function, that is $f$ is $-f$ on the left side. Now the orientation of $\gamma$ is reversed and so $c$ is a positive Morse puncture and $q$ is a negative mixed puncture. Since the order of augmentations is reversed , the disk contributes to $\sigma_{-,n-k}$.
	\end{proof}
	
	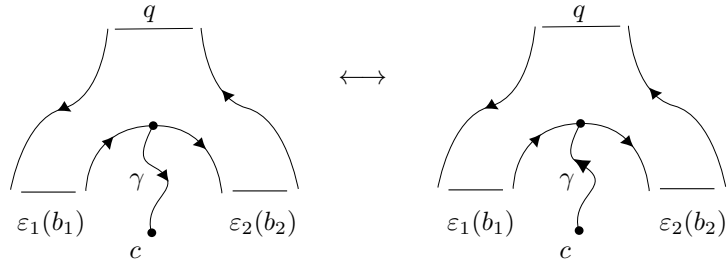
\begin{figure}[!h]
		\centering
		\setlength{\unitlength}{0.1\textwidth}
		\begin{picture}(10,2.5)
			\put(1.5,0){\includegraphics[scale=1]{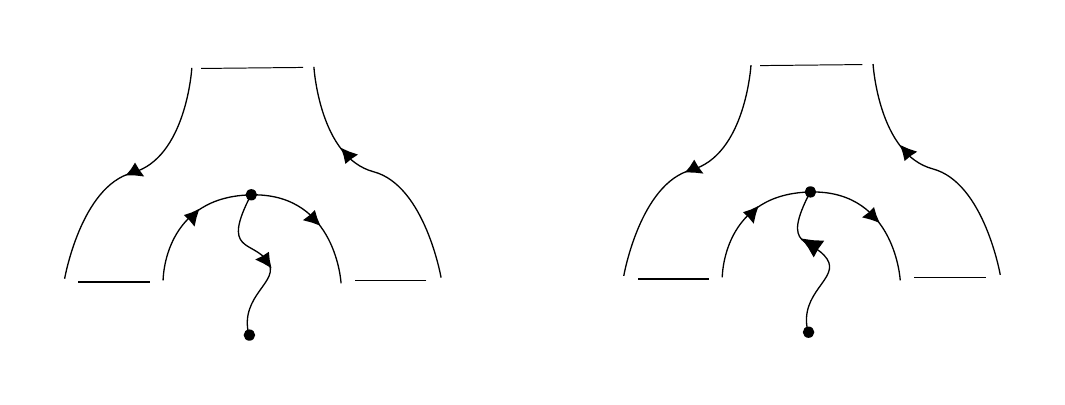}}
			\put(2.0,0.5){$\epsilon_1(b_1)$}
			\put(3.89,0.5){$\epsilon_2(b_2)$}
			\put(5.83,0.5){$\epsilon_1(b_1)$}
			\put(7.72,0.5){$\epsilon_2(b_2)$}
			\put(3.12,2.4){$q$}
			\put(6.95,2.4){$q$}
			\put(3.00,0.25){$c$}
			\put(6.83,0.25){$c$}
			\put(3.00,0.9){$\gamma$}
			\put(6.83,0.9){$\gamma$}
			\put(4.86,1.8){$\longleftrightarrow$}
		\end{picture}
		\caption{Effect of passing from homology to cohomology on a generalized lifted disk.}
		\label{change of sign disk}
	\end{figure}

	\subsection{The map $\tau_n$}
	
	Let us focus on the $n$-th level of the duality long exact sequence. Consider a Reeb chord $a$ with grading $n$, than the differential $\partial$ counts $(u,\gamma)$ lifted generalized disks where $u$ is a pseudo-holomorphic curve and $\gamma$ is the gradient flow line from a generic point of the boundary to $m_j$ a maximum of the Morse function $f$ on $\L_j$ the corresponding component of $\Lambda$, this in particular implies that the beginning point of $\gamma$ on the boundary must be the maximum $m_j$. For $l\in \{1,\ldots,m+1\}$ denote this moduli space by $$\mathcal{M}_{a;\textbf{b},\textbf{p}}(a;b_1,\ldots,b_{l-1},p_l,b_{l},\ldots,b_m),$$ where $p_l$ lies on the boundary component of the punctured disk which was mapped to $\L_{j}$ where $j=j_{l}$ assuming $b_{l}$ is a Reeb chord from the connected component $\Lambda_{i_l}$ to the connected component $\Lambda_{j_l}$ of the disconnected Legendrian submanifold $L$, for all $l\in \{1,\ldots,m\}$ if $l-1=m$ then $j=i_m$. We denote by $[\L_j]$ the homology class which $m_j$ represents. And so by the dimension formula (\cite{dualityseq}, Section 3.3.1)
	$$ 0=\dim(u,\gamma)=\dim \M{a}{b_1}{b_m}+1-\operatorname{Index}_f(p_l)=|a|-|\textbf{b}|-1+1-n$$
	and the fact that $|a|=n$ we get that $|\textbf{b}|=0$.

	\begin{figure}[!h]
		\centering
		\setlength{\unitlength}{0.1\textwidth}
		\begin{picture}(10,2.3)
			\put(2.1,0){\includegraphics[scale=0.5]{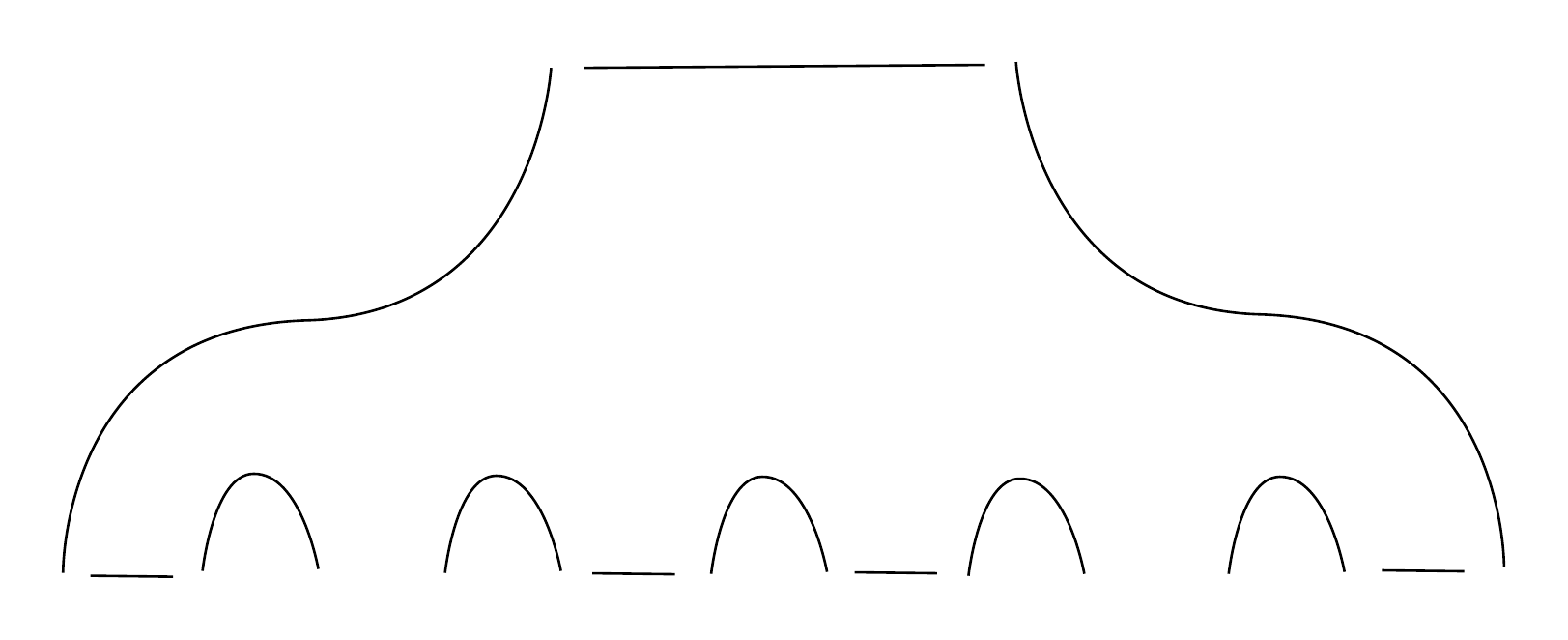}}
			\put(2.48,-0.1){$b_1$}
			\put(4.21,-0.1){$b_{l-1}$}
			\put(5.21,-0.1){$b_{l}$}
			\put(7.03,-0.1){$b_m$}
			\put(6.04,0.4){$\dots$}
			\put(3.32,0.4){$\dots$}
			\put(4.8,2.1){$a$}
			\put(4.73,0.51){$\bullet$}
			\put(4.73,0.75){$p_l$}
		\end{picture}
		\caption{Pointed disc.}
		\label{pointed disk}
	\end{figure}
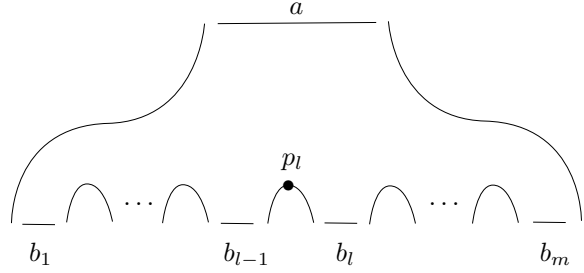
	This gives us the description of action the map $\tau_n$ on $a$ that is
	\begin{equation}\label{tau}
		\tau_n(a)=\sum_{|\textbf{b}|=0} \#_2\mathcal{M}_{a;\textbf{b},\textbf{p}}\sum_{j=1}^{m}\epsilon_1(b_1)\ldots\epsilon_1(b_{l-1})[\L_{i_{l-1}}]\epsilon_2(b_{l})\ldots\epsilon_2(b_m)
	\end{equation}
	
	\subsection{Effect of Legendrian ambient $0$-surgery}
	Let $r\in \mathbb{N}$. Consider disconnected Legendrian submanifold $\Lambda=\coprod_{j=1}^r\L_j$. Denote by $\Lambda_{S,1}$ the submanifold resulting from the Legendrian ambient $0$-surgery by connecting $\L_1$ and $\L_2$. Now inductively $\Lambda_{S,k}$ denotes the submanifold resulting from the Legendrian ambient $0$-surgery by connecting $\Lambda_{S,k-1}$ and $\L_k$ for $k=3,\ldots, r$. 
	
	Now we will restrict the setting to the first iteration of the $0$-surgery for simplicity. The effect of the ambient Legendrian $0$-surgery using the embedded sphere $\textbf{S}^0$ into $\L_i$ and $\L_j$ for some $i,j\in \{1, \ldots,r\}$ on the Chekanov-Eliashberg algebra $\CE{\Lambda}$ was described by Dimitroglou Rizell in Section 1 of \cite{Surgery}, more specifically, for embedded spheres of all dimensions from $0$ to $n-1$. Denote by $\Lambda_{S}$ the Legendrian submanifold obtained by performing the surgery. The situation is as follows: 
	
	The algebra $\CE{\Lambda_S}$ is isomorphic to the algebra $\CE{\Lambda,S}$ defined as the free product of the algebra $\CE{\Lambda}$ and the line $\Z_2\langle s\rangle$ of a formal generator $s$ which corresponds to $c_S$ a new Reeb chord of $\Lambda_S$ that is of degree $|c_S|=n-1$. Note that this means that for any $\epsilon$ augmentation of $\CE{\Lambda_S}$ we have that $\epsilon(c_S)=0$. The differential on $\CE{\Lambda,S}$ is to be denoted by $\partial_S$ and it decomposes into $\partial_S=\partial+h$ on generators. Here $\partial$ is the differential of $\CE{\Lambda}$ and \begin{equation}
		h(a)=\sum_{|a|-|\textbf{b}|-|\textbf{s}|=1}|\Mc{a}{b_1}{b_m}{a;\textbf{b},\textbf{w}}|s^{w_1} b_1 \ldots b_m s^{w_{m+1}}
	\end{equation} counts number of holomorphic disks with boundary on $L$ and with $w_i$ marked points on the corresponding part of boundary of the disk that is mapped to one of the points that are in the image of $\textbf{S}^0$. Here $|\textbf{s}|=(w_1+\ldots+w_{m+1})(n-1)$. For more details see (\cite{Surgery}, Section 6).
		\begin{figure}[!h]
			\centering
			\setlength{\unitlength}{0.1\textwidth}
			\begin{picture}(10,2.5)
				\put(2.1,0){\includegraphics[scale=0.5]{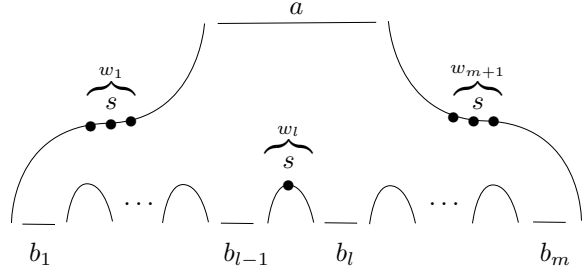}}
				\put(2.48,-0.1){$b_1$}
				\put(4.21,-0.1){$b_{l-1}$}
				\put(5.21,-0.1){$b_{l}$}
				\put(7.03,-0.1){$b_m$}
				\put(6.04,0.4){$\dots$}
				\put(3.32,0.4){$\dots$}
				\put(4.8,2.1){$a$}
				\put(4.73,0.51){$\bullet$}
				\put(4.58,0.73){$\overbrace{s}^{w_{l}}$}
				\put(2.97,1.04){$\bullet$}
				\put(3.33,1.08){$\bullet$}
				\put(3.15,1.06){$\bullet$}
				\put(3.0,1.25){$\overbrace{s}^{w_{1}}$}
				\put(6.2,1.12){$\bullet$}
				\put(6.56,1.08){$\bullet$}
				\put(6.38,1.08){$\bullet$}
				\put(6.23,1.25){$\overbrace{s}^{w_{m+1}}$}
			\end{picture}
			\caption{Twisted disc.}
			\label{twisted disk}
		\end{figure}

		Consider $\epsilon_1,\epsilon_2$ two augmentations of $\CE{\Lambda}$ then $\epsilon^\prime_1,\epsilon^\prime_2$ are two augmentations of $\CE{\Lambda,S}$ induced in the as the pull-back. More specifically, they both vanish on the element $s$ and coincide with the original augmentations of original Reeb chords. Now $\partial_S^{\epsilon_1,\epsilon_2}$ the bilinearized differential decomposes  $\partial_S^{\epsilon_1,\epsilon_2}=\partial^{\epsilon_1,\epsilon_2}+h^{\epsilon_1,\epsilon_2}$ on generators.
		
		If one of $w_j>1$ or if there are two $j\neq j^\prime$ such that $w_j\neq 0$ and $w_{j^\prime}\neq 0$, then the corresponding disk contributes by zero to the bilinearized differential. It must hold that there is exactly one $j\in \{1\,\ldots,m+1\}$ such that $w_j=1$ for a disk to contribute to $h^{\epsilon_1,\epsilon_2}$. Otherwise, it has already contributed to the usual bilinearized differential $\partial^{\epsilon_1,\epsilon_2}$. Let us denote by $\rho_\bullet=h^{\epsilon_1,\epsilon_2}_\bullet$
		\begin{equation*}\label{rho}
			\rho_\bullet=\sum_{|a|-|\textbf{b}|=n}|\Mc{a}{b_1}{b_m}{a;\textbf{b},\textbf{w}}|\sum_{l=1}^{m} \epsilon_1(b_1)\ldots \epsilon_1(b_{l-1})\, s\, \epsilon_2(b_{l}) \ldots\epsilon_2(b_m).
		\end{equation*} 
		
		Let $\epsilon_1,\epsilon_2$ be two augmentations of $\CE{\Lambda}$ over $\Z_2$. Denote by $\epsilon_1^S,\epsilon^S_2$ the induced augmentations of $\CE{\Lambda,S}$.
		
		The inclusion of the line $\Z_2\langle s \rangle_{n-1}$ into $C_\bullet(\Lambda,S)$ makes it into a subcomplex and the fact  that $C_\bullet(\Lambda,S)/\Z_2\langle s \rangle\cong C_\bullet(\Lambda)$ now yields a short exact sequence of complexes 
		\begin{equation*}
			0\to (\Z_2\langle s\rangle_\bullet, \partial^{\epsilon^S_1,\epsilon^S_2}_S)\to (C_\bullet(\Lambda,S),\partial^{\epsilon^S_1,\epsilon^S_2}_S)\xrightarrow{\pi} (C_\bullet(\Lambda),\partial^{\epsilon_1,\epsilon_2})\to 0
		\end{equation*}
		this induces the following long exact sequence in homology
		\begin{equation*}
			\dots\to \lbsch{1}{2}{\bullet}{\Lambda_S}\xrightarrow{\pi_\bullet}\lbch{1}{2}{\bullet}{\Lambda}\xrightarrow{\rho_\bullet}\Z_2\langle s\rangle_{\bullet-1}\to \lbsch{1}{2}{\bullet-1}{L_S}\to \dots.
		\end{equation*}
		Since $\Z_2\langle s\rangle_{\bullet-1}=0$ if $\bullet\neq n$ we obtain the isomorphism:
		\begin{equation*}
			0\to \lbsch{1}{2}{\bullet}{\Lambda_S}\xrightarrow{\pi_\bullet}\lbch{1}{2}{\bullet}{\Lambda}\xrightarrow{\rho_\bullet}0.
		\end{equation*}
		
		\section{Proof of Theorem \ref{main thm}}
		
		First, let us say why we formulate Theorem \ref{main thm} using the map $\tau_n$ and not a map $\tau_0$ like it is done in Theorem \ref{thm: criterBG}, that is,
		\begin{equation}\label{eq: homotcrit}
			\epsilon_1\sim\epsilon_2 \iff \tau_0=0.
		\end{equation}
		To prove \eqref{eq: homotcrit} one has to show that it holds in homology that $$\epsilon_1(\cdot)-\epsilon_2(\cdot)=\tau_0(\cdot)\bullet [\Lambda],$$ where the pairing with the fundamental class $\cdot\bullet[\Lambda]: H_0(\L)\to \Z_2$ is an vector space automorphism of $\Z_2$. The isomorphism has purely formal character in the connected case, however, in the disconnected case, this is no-longer and isomorphism and formulation \eqref{eq: homotcrit} fails even in the following basic example of a disconnected Legendrian submanifold.
		
		As in \cite{geog}
		let us denote by $\L^{(2)}$ the standard $n$-dimensional Legendrian Hopf link in $J^1(\mathbb{R}^n)$ so that the Maslov potential on the upper component is the Maslov potential of the lower component enlarged by 1.
		
		\begin{figure}[!h]
			\centering
			\includegraphics[width=40mm]{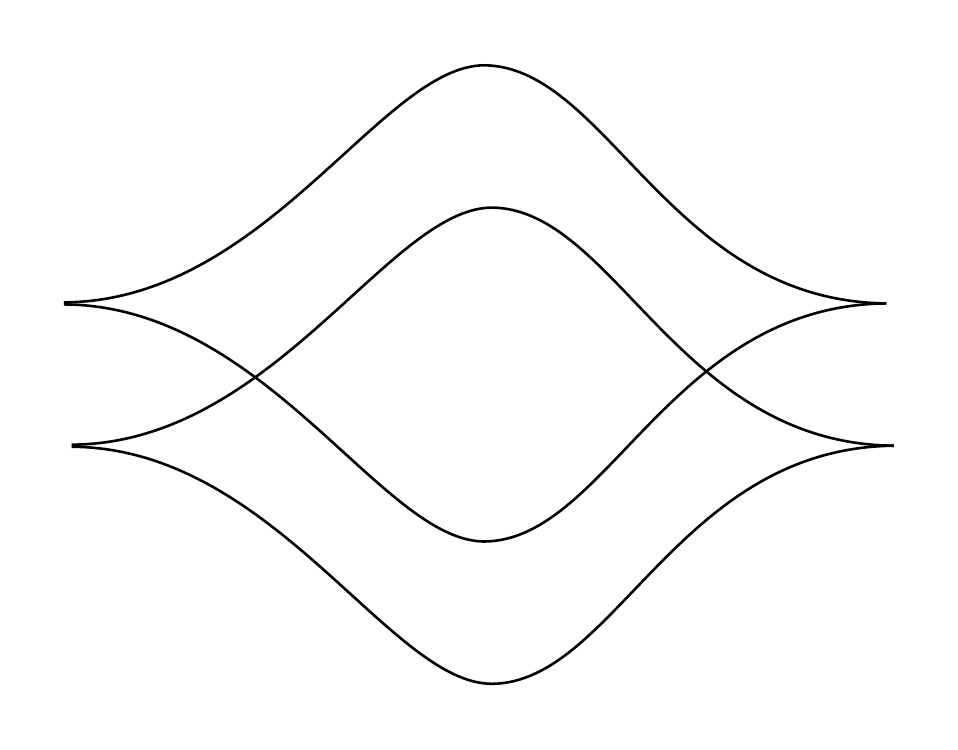}
			\caption{Front projection of the Hopf link $\L^{(2)}$ when $n=1$.}
			\label{fig:hopf link}
		\end{figure}
		
		The Corollary 4.7 in \cite{geog} then implies that there are two augmentations $\epsilon_L$ and $\epsilon_R$ of Chekanov-Eliashberg algebra of $\L^{(2)}$ so that $\epsilon_L(m_{12})=1$ and $\epsilon_R(m_{12})=0$ for a chord $m_{12}$ and they vanish otherwise. 
		
		Therefore, posing $\epsilon_1=\epsilon_R,\epsilon_2=\epsilon_R$, the map $\tau_0$ sends each Reeb chord to zero except for the chord $m_{12}$ which is send to the diagonal of $H_0(\L)\oplus H_0(\L)\simeq\Z_2\oplus\Z_2$. And so, in this case  $\tau_0\neq 0$ even though $\epsilon_1\sim\epsilon_2$.
		
		From now on, $\L$ denotes a disconnected Legendrian submanifold with $r$ connected closed components and $\Lambda_S$ a connected Legendrian submanifold obtained by performing $r-1$ Legendrian ambient surgeries on $\L$. And by $\epsilon_1^S,\epsilon_2^S$ denote the augmentations induced by this surgery so that they vanish on the formal generators added to the Chekanov-Eliashberg algebra and coincide with $\epsilon_1$ and $\epsilon_2$ otherwise. This implies that on chords of $\Lambda$ are decorated with the same number by  both $\epsilon_i^S$ and $\epsilon_i$ for both $i=1,2$. Considering Lemma 3.1 from \cite{geog} we obtain the following proposition
		
		\begin{prop}\label{S epsilon lemma}
			$\epsilon_1\sim \epsilon_2$ if and only if $\epsilon_1^S\sim \epsilon_2^S$
		\end{prop}

		\begin{lemma}\label{diagram alpha}
			Consider the following diagram:
			\begin{equation}
				\begin{tikzcd}
					\lbsch{1}{2}{0}{\Lambda_S}\arrow{rr}{\pi^+}\arrow{d}{\tau_{+,0}^S}& & \lbch{1}{2}{0}{\Lambda}\arrow{d}{\tau_{+,0}}\\
					\ho{0}{\Lambda_S}\arrow{dr}[below left]{\gamma} & &\ho{0}{\Lambda} \arrow{ll}[above right]{\widetilde{\alpha}}\arrow{ld}{\alpha}\\
					&\Z_2 &  
				\end{tikzcd}
			\end{equation}
			where $\widetilde{\alpha}:\ho{0}{\Lambda}\to \ho{0}{\Lambda}$ is defined as follows. Denote by $[*_{\L_j}]$ for $j=1,\ldots,r$ the classes of points in $\ho{0}{\Lambda}$ that represent distinct components $\L_j$ of $\Lambda$, and $[*_{\Lambda_S}]\in\ho{0}{\Lambda_S}$ a class of point in $\Lambda_S$, that is connected. Moreover, consider the map $\gamma:\ho{0}{\Lambda_S}\to \Z_2$ defined as $a[*_{\Lambda_S}]\to a$ for any $a\in \Z_2$, and let us denote by $\alpha:\ho{0}{\Lambda}\to \Z_2$ the composition $\gamma\circ \widetilde{\alpha}$, i.e. $\widetilde{\alpha}(\bigoplus_{j=1}^ra_j[*_{\L_j}])\mapsto \Big(\sum_{j=1}^r a_j\Big) [*_{\Lambda_S}]$ for any $a_j\in \Z_2$. The map $\pi^+$ denotes the composition of isomorphisms in level zero of the surgery long exact sequence.
			Then the diagram above commutes.
		\end{lemma}
		\begin{proof}
			Let $m_j$ be minimum of the perturbing Morse function on the component $\L_{j}$ if $q$ is a chord of degree $0$ that starts on $\L_{i}$ and ends on $\L_{j}$, thanks to the rigidity of the formal disk the starting point of a generalized lifted disk must map either to the starting point of $q$ or to the ending point of it.
			In the first case, the disk contributes with $\epsilon_2(q)[*_{\L_{i}}]$ to $\tau_{+,0}(q)$. In the second case,  the disk contributes with $\epsilon_1(q)[*_{\L_{j}}]$ to $\tau_{+,0}(q)$. And so $\tau_{+,0}(q)=\epsilon_1(q)[*_{\L_{j}}]+\epsilon_2(q)[*_{\L_{i}}]$. By the proof of Proposition 3.2 in \cite{geog} we have that $\tau_{+,0}^S=\epsilon_1+\epsilon_2$. Now it is clear that for each $q$ the diagram commutes and so it commutes.
		\end{proof}
		\begin{lemma}\label{form of alpha}
			For each $c\in\ho{0}{\Lambda}$ and the map $\alpha:\ho{0}{\Lambda}\to \Z_2$ from the statement of Lemma \ref{diagram alpha} it holds that \begin{equation}
				\alpha(c)=c\bullet [\Lambda].
			\end{equation} 
		\end{lemma}
		\begin{proof}
			The map $\alpha$ is an element of $(\ho{0}{\Lambda})^*$ of the dual of $\ho{0}{\Lambda}$. The intersection pairing $\bullet:\ho{0}{\Lambda}\otimes_{\Z_2} \ho{n}{\Lambda}\to\Z_2$ defines $\Theta:(\ho{0}{\Lambda})^*\to\ho{n}{\Lambda}$ an isomorphism that sends $\delta\in (\ho{0}{\Lambda})^*$ to a class $\Theta(\delta)\in \ho{n}{\Lambda}$ so that $p\bullet \Theta(\delta)=\delta(p)$ for each $p\in \ho{0}{\Lambda}$.
			
			Let us claim that $\Theta(\alpha)=[\Lambda]$. Then our claim is equivalent to the statement that for every $p\in \ho{0}{\Lambda}$ it holds that $\alpha(p)=p\bullet[\Lambda]$. Since elements $e_j\in \ho{0}{\Lambda}$ with only non-zero component equal to $[*_{\L_{j}}]$ generate the space $\ho{0}{\Lambda}$ and clearly $\alpha(e_j)=1$. By the definition of the pairing $\bullet$ for the disconnected Legendrian submanifold $\Lambda$ we have that
			\begin{equation}
				e_j \bullet [\Lambda]=[*_{\L_{j}}]\bullet[\Lambda]+\sum_{i=1, i\neq j}^{r}0\bullet [\L_1]=[*_{\L_{j}}]\bullet[\Lambda]=1
			\end{equation}
			where the last component is by the Poincaré duality for closed component $\L_j$ and so $([*_{\L_{1}}]\oplus 0) \bullet ([\L_1]\oplus[\L_2])=1$. The reasoning for the other generating classes $e_j$ is analogous.
		\end{proof}
		
		\begin{proof}[Proof of Theorem \ref{main thm}]
			
			First, consider $\epsilon_1\not\sim \epsilon_2$. Then $\epsilon_1^S\not\sim\epsilon^S_2$ by Proposition \ref{S epsilon lemma}. Thanks to Proposition 3.2 in \cite{geog} we have that $\tau_{+,0}^S\neq 0$. For the sake of contradiction suppose that $[\Lambda]\in \im_{\Z_2}\tau_{-,n}$, then by the exactness of the duality sequence for the negative order of augmentations we obtain that $\sigma_{-,0}([\Lambda])=0$. And so by Proposition \ref{signs switching} we obtain that
			\begin{equation}
				0=\langle\sigma_{-,0}([\Lambda]),q\rangle_-=\tau_{+,0}(q)\bullet [\Lambda]
			\end{equation} 
			for every $q$ Reeb chord that gives rise to a generator of $\lbch{2}{1}{n}{\Lambda}$ and a generator of $\lbch{1}{2}{n}{\Lambda}$. However, using the commutativity from Lemma \ref{diagram alpha} and the form of the map $\alpha$ from Lemma \ref{form of alpha} we obtain that for every Reeb chord $q$
			\begin{equation}
				\gamma\circ\tau^S_{+,0}\circ(\pi^+_0)^{-1}(q)=\alpha\circ\tau_{+,0}(q)=\tau_{+,0}(q)\bullet [\Lambda]=0
			\end{equation} 
			which yields that $\tau_{+,0}^S=0$ because both $\gamma$ and $\pi^+_n$ are isomorphisms. This is the contradiction with $\tau^S_{+,0}\neq 0$.

			On the other hand, assume that $\epsilon_1\sim\epsilon_2$, then by Proposition \ref{S epsilon lemma} we have that $\epsilon_1^S\sim \epsilon_2^S$, which yields $\tau_{+,0}^S=0$ as above. Suppose that $[\Lambda]\not\in \im_{\Z_2}\tau_{-,n}$, then $\sigma_{-,0}([\Lambda])\neq 0$, and thanks to Proposition \ref{prop: almost non-degeneracy}, there exists a chord $q$ so that $\langle \sigma_{-,0}([\Lambda]),q\rangle_-\neq 0$.
			
			If $\tau_{+,0}=0$, then $ 0\neq\langle\sigma_{-,0}([\Lambda]),q\rangle_-=\tau_{+,0}(q)\bullet [\Lambda]=0$ which is a contradiction.
			
			If $\tau_{+,0}\neq 0$, then the commutativity of the diagram in Lemma \ref{diagram alpha}, Lemma \ref{form of alpha}, and the fact that $\tau_{+,0}^S=0$ imply that \begin{equation}
				0=\gamma\circ\tau^S_{+,0}\circ(\pi^+_0)^{-1}(q)=\alpha\circ\tau_{+,0}(q)=\tau_{+,0}(q)\bullet [\Lambda]=\sigma_{-,0}([\Lambda]),q\rangle_-\neq 0
			\end{equation}
			which is a contradiction. This completes the proof.
			
		\end{proof}

		Consider a disconnected Legendrian submanifold $\Lambda$ as above. We want to show that there are no further obstructions regarding the dimension of the image of the map $\tau_{-,n}$ or map $\tau_{+,n}$. In the next section, we will use this result to prove that there is not any other obstruction on the DGA-homotopy of the given augmentation of Chekanov-Eliashberg algebras of Legendrian submanifolds having $n$-spheres as its connected components, since this amounts to the discussion of the geography of bilinearized Legendrian contact homology for such submanifolds.

		\begin{prop}\label{geography}
			For any integer $r\geq 2$ and any non-negative integer $m<r$ there exists a disconnected Legendrian submanifold $\Lambda^{r}$ and augmentations $\epsilon_L^m,\epsilon_R^m$ of its Chekanov-Eliashberg algebra so that $\dim_{\Z_2}\im\tau_{+,n}=m$.
		\end{prop}
		\begin{proof}
			Perform a Legendrian ambient surgery on those two components of $\L^{(2)}$ producing a Legendrian $\L^\prime$.
			
			\begin{figure}[!h]
				\centering
				\setlength{\unitlength}{0.1\textwidth}
				\begin{picture}(10,2)
					\put(0,0){\includegraphics[width=\textwidth]{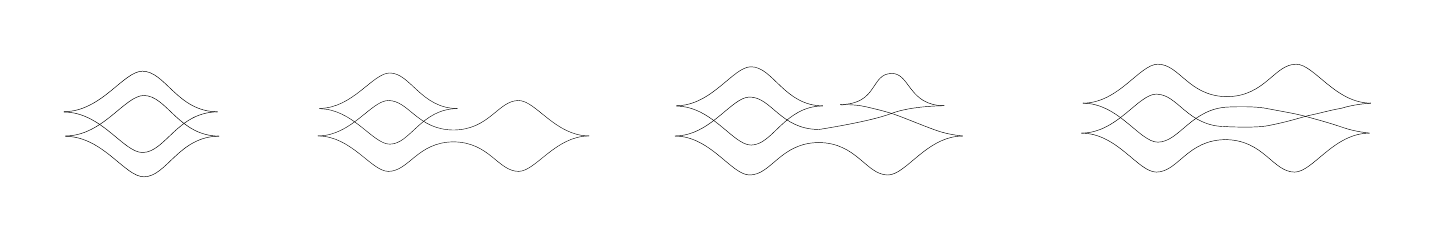}
					}
				\end{picture}
				\caption{The construction of the Legendrian submanifold $\Lambda^\prime$ when $n=1$.}
				\label{fig:isotopy for surgery}
			\end{figure}
			
			Pull-back the augmentations $\epsilon_L$ and $\epsilon_R$ onto the algebra $\CE{\L^\prime}$. Consequently Proposition 3.2 from \cite{geog} implies that those two pull-backed augmentations $\widetilde{\epsilon_L}$ and $\widetilde{\epsilon_R}$ are not DGA-homotopic since $\partial (m_{12})=0$ by Proposition 4.5 and $\tau_{+,0}(m_{12})=\widetilde{\epsilon_L}(m_{12})-\widetilde{\epsilon_R}(m_{12})=1\neq 0$. The fact \ref{odkaz na jedn} yields that $\dim_{\Z_2}\im\tau_{+,n}=0$.
			
			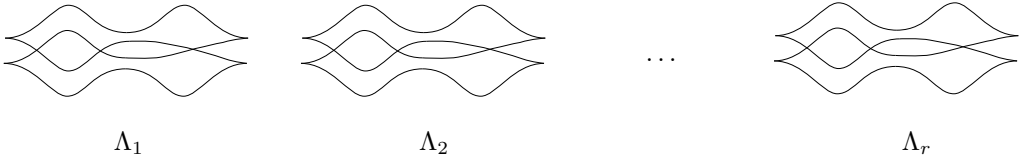
\begin{figure}[!h]
				\centering
				\setlength{\unitlength}{0.1\textwidth}
				\begin{picture}(10,2)
					\put(0,0){\includegraphics[width=\textwidth]{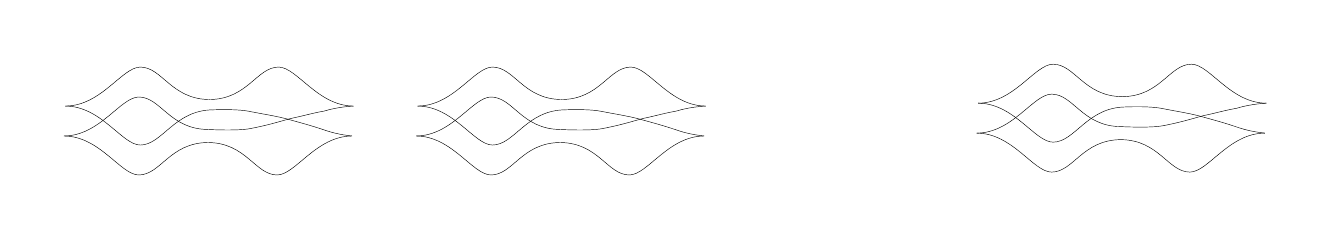}
					}
					\put(1.46,0){$\Lambda_1$}
					\put(4.18,0){$\Lambda_2$}
					\put(8.48,0){$\Lambda_r$}
					\put(6.2,0.8){$\dots$}
				\end{picture}
				\caption{Front projection of the Legendrian $\L^r$  when $n=1$.}
				\label{fig:minimal link}
			\end{figure}
			
			The disconnected Legendrian submanifold $\L^r=\coprod_{j=1}^r\L_j$ is defined as $r$ unlinked horizontally displaced copies $\L_{j}$ of $\L^\prime$. That means that the bilinearized complex splits into $r$ copies because there are no Reeb chords among distinct components.
			In particular,\begin{equation*}
				\tau^{\L_p}_{+,n}=\bigoplus_{j=1}^r\tau^{\L_{j}}_{+,n}
			\end{equation*} that is for the following augmentations the rank of the resulting map $\tau^{\L^r}_{+,n}:LCH_n^{\epsilon_L^m,\epsilon_R^m}(\L^r)\to \ho{n}{\L^r}=\bigoplus_{j=1}^r\ho{n}{\L_j}$ is the sum of the ranks of the maps $\tau^{\L_{j}}_{+,n}:LCH_n^{\widetilde{\epsilon_L},\widetilde{\epsilon_R}}(\L_{j})\to\ho{n}{\L_{j}}$ playing the same role as $\tau^{\L^r}_{+,n}$ but in the duality exact sequence of the corresponding component.
			Fix some $m\in \{ 0,\ldots, r-1\}$. Now define the augmentations by the assignment for any $c$ a chord of $\L_{j}$ $
			\epsilon_L^m(c)=\widetilde{\epsilon_L}(c)
			$ and if $m=0$, then $\epsilon_R^m(c)=\widetilde{\epsilon_R}(c)$, otherwise
			\begin{equation*}
				\epsilon_R^m(c)=\begin{cases}
					\widetilde{\epsilon_L}(c); & 1\leq j\leq m,\\
					\widetilde{\epsilon_R}(c); & m<j\leq r.
				\end{cases}
			\end{equation*}
			By the construction of $\epsilon_L^m$ and $\epsilon_R^m$ it is clear that since $m\neq r$ they are not DGA-homotopic because otherwise we could factor the corresponding  $(\epsilon_L^m,\epsilon_R^m)$-derivative through the chords of $r$-th component $\L^{r}$, which is impossible, and thus $$\dim_{\Z_2}\im\tau^{\L^r}_{+,n}=\sum_{j=1}^r \dim_{\Z_2}\im\tau^{\L_{j}}_{+,n}=m\cdot1+(r-m)\cdot 0=m$$ as we desired.

		\end{proof}
		
		\section{Proof of Theorem \ref{geography1} and Theorem \ref{geography2}}
		
		For a Legendrian submanifold $\Lambda$ and two non-homotopic augmentations $\epsilon_1,\epsilon_2$ let us denote
		\begin{equation}
			P^+=P_{\L,\epsilon_1,\epsilon_2}(t)=\sum_{k\in\Z} \dim_{\Z_2} LCH^{\epsilon_1,\epsilon_2}_{k}(\L)\,\, t^k.
		\end{equation}
		and similarly $P^-$ for the opposite ordering of augmentations. Those Laurent polynomial split as
		$P^\pm=\textbf{p}^\pm+\textbf{q}^\pm$, where $\textbf{q}^\pm(t)=\sum_{k\in \Z}\dim_{\Z_2}\im \tau_{\pm,k}$ and $\textbf{p}^\pm(t)=\sum_{k\in \Z}\dim_{\Z_2}\ker \tau_{\pm,k}$.
		
		Similarly to \cite{geog} we can define the following.
		
		\begin{definition}\label{def: admissible polyn}
			Let $n$ be natural numbers and $P\in \N_0[t,t^{-1}]$ be a Laurent polynomial so that $P=\textbf{q}+\textbf{p}$. We say that $P$ is lbLCH-admissible if:
			\begin{enumerate}[label=(\roman*)]
				\item $\textbf{q}\in\N_0[t]$ is a polynomial, $\deg(q)\leq n$, and $\textbf{q}(0)\geq 1$,
				\item if $n$ is odd, then $\textbf{p}(-1)$ is even, and if $n$ is even, then $\textbf{p}(-1)=0$. 
			\end{enumerate}
		\end{definition}

		\begin{proof}[Proof of Theorem \ref{geography1}]
			
			To prove (i), $\im\tau_{\pm,k}\subset \ho{k}{L}$ and so $\textbf{q}^\pm\in\Z[t]$ and $\deg(\textbf{q}^\pm)\leq n$. Theorem \ref{main thm} and the fact that $\epsilon_1\not\sim\epsilon_2$ imply that the class $[L]$ is not in the image of $\tau_{\pm,n}$. Therefore,  $q_n^\pm=\dim_{\Z_2}\im\tau_{\pm,n}<r$. By adjointness, we know that $\dim_{\Z_2}\sigma_{\mp,n}<r$, thus $\dim_{\Z_2}\ker\tau_{\mp,0}=r-\dim_{\Z_2}\sigma_{\mp,n}\geq 1$. 
			
			The proof of (ii) coincides with the proof of Proposition 4.2 in \cite{geog}. Therefore, both $P_{\L,\epsilon_1,\epsilon_2}$ and $P_{\L,\epsilon_2,\epsilon_1}$ are lbLCH-admissible.
		\end{proof}
		
		\begin{proof}[Proof of Theorem \ref{geography2}]
			Let us have arbitrary $P=\textbf{q}+\textbf{p}$ lbLCH-admissible Poincaré polynomial.
			
			For the Hopf link $\L^{(2)}$ we have that $P_{\L^{(2)},\epsilon_L,\epsilon_R}=1+t^n$. And because by Proposition 3.5 in \cite{geog} for non-homotopic augmentations $\tilde{\epsilon}_L,\tilde{\epsilon}_R$ the connected sum acts as subtraction of the term $t^n$ we obtain that for the disconnected Legendrian submanifold $\L^m$ from Proposition \ref{geography} for $m=\textbf{q}(0)-1$ its Poincaré polynomial is the following constant
			\begin{equation}\label{poincaré pol}
				P_{\L^{m},\epsilon_L^{m},\epsilon_R^{m}}=m.
			\end{equation}  Note that if a link consists of multiple components so that their projections into $T^*M$ does not intersect, then the Poincaré polynomial is given by the sum of Poincaré polynomials of corresponding components.
			
			Now choose the Laurent polynomials  $\tilde{\textbf{p}},\tilde{\textbf{q}}$ with non-negative coefficients so that
			\begin{gather}
				\textbf{p}=\tilde{\textbf{p}}\,\, \text{ and }\,\, \textbf{q}-\textbf{q}(0)-q_nt^n=\tilde{\textbf{q}}-1,
			\end{gather}
			 and moreover, let  $\Psi_{\tilde{\textbf{q}},\tilde{\textbf{p}}}=\L^{(2N)}_{p}$ be the connected Legendrian submanifold in $J^1(M)$ and $\bar{\epsilon}_L,\bar{\epsilon}_R$ the augmentations of its Eliashberg-Chekanov algebra so that $P_{\Psi_{\tilde{\textbf{q}},\tilde{\textbf{p}}},\bar{\epsilon}_L,\bar{\epsilon}_R}=\tilde{\textbf{q}}+\tilde{\textbf{p}}$ that is due to Bourgeois and Galant (see \cite{geog}).
			
			Consider the disconnected Legendrian submanifold $\L$ that consists of the submanifold $\Psi_{\tilde{\textbf{q}},\tilde{\textbf{p}}}$, $q_n$ horizontally displaced copies of $\L_{W}$, the $n$-dimensional Legendrian lift of the Whitney immersion with the single augmentation $\epsilon_w$ and polynomial $P_{\L_W,\epsilon_w}=t^n$, and the disconnected Legendrian submanifold $\L^{m}$ with the following augmentations:
			\begin{equation}
				\epsilon_1=\begin{cases}
					\epsilon_L^{m}; & \text{ on chords of } \L^{m},\\
					\bar{\epsilon}_L; & \text{ on chords of } \Psi_{\tilde{\textbf{q}},\tilde{\textbf{p}}},\\
					\epsilon_w; & \text{ on chords of } \L_{W},
				\end{cases}
			\end{equation}
			and for $\epsilon_2$ analogously with $R$ and $L$ exchanged. Observe that since the Chekanov-Eliashberg algebra of the Legendrian $\L$ splits, then the DGA-homotopy descends to the components, but on the copy corresponding to $\Psi_{\tilde{\textbf{q}},\tilde{\textbf{p}}}$ the augmentations $\epsilon_{1}$ and $\epsilon_{2}$ are not DGA-homotopic. The claim that the disconnected Legendrian submanifold $\L$ has the Poincaré polynomial equal to $P$ easily follows
			\begin{equation}
				P_{\L,\epsilon_{1},\epsilon_{2}}=m+\tilde{\textbf{p}}+\tilde{\textbf{q}}+q_nt^n=\textbf{q}(0)-1+\tilde{\textbf{p}}+\tilde{\textbf{q}}+q_nt^n=\textbf{p}+\textbf{q}=P.
			\end{equation}

		\end{proof}

		\bibliographystyle{unsrt}     
		\bibliography{GeographyversionII.bib}

\begin{thebibliography}{10}

\bibitem{ChekanovLCH}
Yuri Chekanov.
\newblock Differential algebra of legendrian links.
\newblock {\em Inventiones Mathematicae}, volume 150:441--483, Dec 2002.

\bibitem{EEScontacthomology}
Tobias Ekholm, John Etnyre, and Michael Sullivan.
\newblock The contact homology of legendrian submanifolds in
  $\mathbb{R}^{2n+1}$.
\newblock {\em Journal of Differential Geometry}, volume 71(2):177 -- 305,
  2005.

\bibitem{bil}
Frédéric Bourgeois and Baptiste Chantraine.
\newblock {Bilinearized Legendrian contact homology and the augmentation
  category}.
\newblock {\em Journal of Symplectic Geometry}, volume 12(3):553 -- 583, 2014.

\bibitem{dualityseq}
Tobias Ekholm, John~B. Etnyre, and Joshua~M. Sabloff.
\newblock {A duality exact sequence for legendrian contact homology}.
\newblock {\em Duke Mathematical Journal}, volume 150(1):1 -- 75, 2009.

\bibitem{geog}
Frédéric Bourgeois and Damien Galant.
\newblock Geography of bilinearized legendrian contact homology.
\newblock preprint, https://arxiv.org/abs/1905.12037, 2020.

\bibitem{EHKmorphism}
Tobias Ekholm, Ko~Honda, and Tamás Kálmán.
\newblock Legendrian knots and exact lagrangian cobordisms.
\newblock {\em Journal of the European Mathematical Society}, volume
  18:2627--2689, Dec 2012.

\bibitem{Karlssonorientation}
Cecilia Karlsson.
\newblock A note on coherent orientations for exact lagrangian cobordisms.
\newblock {\em Quantum Topology}, volume 11, Jul 2017.

\bibitem{gfgeography}
Fr{\'{e} }d{\'{e}}ric Bourgeois, Joshua~M Sabloff, and Lisa Traynor.
\newblock Lagrangian cobordisms via generating families: Construction and
  geography.
\newblock {\em Algebraic \& Geometric Topology}, 15(4):2439--2477, sep 2015.

\bibitem{Ng_2017}
Lenhard Ng, Dan Rutherford, Vivek Shende, and Steven Sivek.
\newblock The cardinality of the augmentation category of a legendrian link.
\newblock {\em Mathematical Research Letters}, volume 24(6):1845--1874, 2017.

\bibitem{Surgery}
Georgios Dimitroglou~Rizell.
\newblock Legendrian ambient surgery and legendrian contact homology.
\newblock {\em Journal of Symplectic Geometry}, volume 14(3):811–901, 2016.

\end{thebibliography}

	\end{document}